\documentclass{elsarticle}

\usepackage[english]{babel}
\usepackage{dsfont} 
\usepackage{graphicx}
\usepackage{color}
\usepackage[hyperindex,breaklinks]{hyperref}
\usepackage{enumerate}
\usepackage{amsmath}
\usepackage{amssymb,empheq} 
\usepackage{amsthm} 
\usepackage{cleveref}
\usepackage[dvipsnames]{xcolor}
\usepackage{subcaption}

\allowdisplaybreaks
\newtheorem{thm}{Theorem}[section]
\newtheorem{lkthm}[thm]{L\'evy--Khinchin Theorem}
\newtheorem{schurthm}[thm]{Schur Product Theorem}
\newtheorem{lem}[thm]{Lemma}
\newtheorem{prop}[thm]{Proposition}
\newtheorem{cor}[thm]{Corollary}

\theoremstyle{definition}

\newtheorem{dfn}[thm]{Definition}
\newtheorem{exm}[thm]{Example}

\theoremstyle{remark}
\newtheorem{rem}[thm]{Remark}

\newcommand{\exmsymbol}{\hfill$\circ$}

\newcommand{\nset}{\mathds{N}}
\newcommand{\pset}{\mathds{P}}

\newcommand{\rset}{\mathds{R}}

\newcommand{\diff}{\mathrm{d}}

\newcommand{\pos}{\mathrm{Pos}}

\newcommand{\supp}{\mathrm{supp}\,}

\newcommand{\id}{\mathrm{id}}
\newcommand{\one}{\mathds{1}}

\newcommand{\stirlingone}[2]{\begin{bmatrix}#1\\#2\end{bmatrix}}
\newcommand{\stirlingtwo}[2]{\begin{Bmatrix}#1\\#2\end{Bmatrix}}
\newcommand{\skl}[2]{\langle #1, #2 \rangle}

\newcommand{\cP}{\mathcal{P}}

\newcommand{\fD}{\mathfrak{D}}
\newcommand{\fd}{\mathfrak{d}}

\author{Philipp J.\ di Dio\footnote{Corresponding author, email: philipp.didio@uni-konstanz.de} and Lars-Luca Langer\footnote{Email: lars-luca.langer@uni-konstanz.de}}
\address{Department of Mathematics and Statistics, University of Konstanz, Universit\"atsstra{\ss}e 10, D-78464 Konstanz, Germany}
\address{Zukunftskolleg, Universtity of Konstanz, Universit\"atsstra{\ss}e 10, D-78464 Konstanz, Germany}

\journal{ArXiv}

\title{The Hadamard Product of Moment Sequences, Diagonal Positivity Preservers, and their Generators}

\begin{document}

\begin{abstract}
In this work we investigate special aspects of positivity preservers and especially diagonal positivity preservers, i.e., linear maps $T:\rset[x_1,\dots,x_n]\to\rset[x_1,\dots,x_n]$ such that $Tx^\alpha = t_\alpha x^\alpha$ holds for all $\alpha\in\nset_0^n$ with $t_\alpha\in\rset$ and $Tp\geq 0$ on $\rset^n$ for all $p\in\rset[x_1,\dots,x_n]$ with $p\geq 0$ on $\rset^n$.
We discuss representations of $T$, give characterizations of diagonal positivity preservers, and compare these to previous (partial) results in the literature.
On the side we get a full characterization of linear maps preserving moment sequences and a new proof of Schur's product formula.
The tool of diagonal positivity preservers simplifies several other existing proofs in the literature.
We give a full characterization of generators $A$ of diagonal positivity preservers, i.e., $e^{tA}$ is a diagonal positivity preserver for all $t\geq 0$.
We give the connection of these generators to infinitely divisible moment sequences.
\end{abstract}


\maketitle

\setcounter{tocdepth}{2}
\tableofcontents

\noindent\rule{\textwidth}{1pt}

\section{Introduction}

Non-negative polynomials
\[\pos(K) := \{ p\in\rset[x_1,\dots,x_n] \,|\, p\geq 0\ \text{on}\ K\}\]
with $n\in\nset$ and $K\subseteq\rset^n$ belong to the most important structures in mathematics and are well-studied \cite{schmud91,marshallPosPoly,bochnak98,prestelPosPoly,
schmudMomentBook,powersPositivityRealPolynomials,scheidererRealAlgebraicGeometry}.
The dual problem is the moment problem, i.e., the study of linear functionals $L\in\pos(K)'$ \cite{akhiezClassical,kreinMarkovMomentProblem,schmud91,lauren09,marshallPosPoly,schmudMomentBook}.
A linear functional $L:\rset[x_1,\dots,x_n]\to\rset$ is called a \emph{$K$-moment functional} if there exists a measure $\mu$ on $K\subseteq\rset^n$ closed such that $L(f) = \int_K f~\diff\mu$ holds for all $f\in\rset[x_1,\dots,x_n]$.
A sequence $s = (s_{\alpha})_{\alpha\in\nset_0^n}$ is called a \emph{$K$-moment sequence} if the \emph{Riesz functional} $L_s(x^\alpha) := s_\alpha$ is a $K$-moment functional.

The second logical step to study $\pos(K)$ is to study all (linear) maps
\[T:\rset[x_1,\dots,x_n]\to\rset[x_1,\dots,x_n],\]
especially when $K$-positivity is preserved:
\begin{equation}\label{eq:Tpospres}
T\pos(K)\subseteq\pos(K).
\end{equation}
Maps with (\ref{eq:Tpospres}) are called \emph{$K$-positivity preservers}.
$\rset^n$-positivity preservers were studied e.g.\ in \cite{guterman08,netzer10,borcea11,didio24posPresConst,didio24posPres2arxiv}.
In \cite{didio24posPres2arxiv} $K$-positivity preservers for all closed $K\subsetneq\rset^n$ are fully characterized for the first time.
In the present manuscript we deal with \emph{diagonal $\rset^n$-positivity preservers}:
\begin{equation}\label{eq:diagPosPres}
Tx^\alpha = t_\alpha x^\alpha \quad\text{for all}\ \alpha\in\nset_0^n\ \qquad\text{and}\qquad T\pos(\rset^n)\subseteq\pos(\rset^n).
\end{equation}
They are fully characterized (\Cref{thm:borceaDiag} \cite[Cor.\ 4.3]{borcea11}), they have a very special structure, and hence they have properties general positivity preservers do not possess.
Despite their special structure diagonal positivity preservers have not been studied in more detail before.
This is the main purpose of the current work.

Positive (semi-)definite sequences, moment sequences, positive kernels, and their infinitely divisibility and maps have been studied before, see e.g.\
\cite{horn69b,horn69a,tyan75,hansen88,sato99,berg04,berg05,bertoin05,berg07,boettcher13,
carcoma17,patie20,blekherman22arxiv,belton22,belton23}.
However, mostly the simpler univariate cases have been studied.
It is known for a long time that positive semi-definiteness in the multivariate moment problem is necessary but not sufficient \cite{schmud79}.
Hence, little about the multivariate cases is known.
We work here solely in the general multivariate setting and explain the difference to previous works in the following sections when appropriate.
The main difference between the current work and the previous works we cited above is that here we use the connection between $K$-moment sequences and $K$-positivitiy preservers \cite{borcea11}, especially the diagonal operator case.
This significantly simplifies arguments and gives new results.

The detailed study here reveals the following and hence the manuscript is structured as follows.
In \Cref{subsec:genNot} we give the general notations, in \Cref{subsec:addConv} we discuss the additive convolution $\mu *\nu$, and in \Cref{subsec:multConv} we give the multiplicative convolution $\mu\odot\nu$ of two measures.
In \Cref{sec:charLinMaps} we give the full characterization of linear maps $S:\rset^{\nset_0^n}\to\rset^{\nset_0^n}$ such that $K$-moment sequences are mapped to $K$-moment sequences.
In \Cref{sec:diagOp} we then start our main topic and discuss first properties and give different representations of diagonal operators.
In \Cref{sec:prodDiag} we look at the product of diagonal positivity preservers.
By (\ref{eq:diagPosPres}) the product of two diagonal positivity preservers is again a positivity preserver and corresponds to the Hadamard product of moment sequences.
With this argument we remove the usage of the \Cref{thm:schur} for the Hadamard product of moment sequences which has been used until now.
This simplifies several arguments in e.g.\ \cite[Thm.\ 3.1 and 9.1]{belton22} and \cite[Sec.\ 2.3]{blekherman22arxiv}.
The argument with diagonal operators is so strong that we prove the \Cref{thm:schur} with diagonal operators in \Cref{sec:schur}.
In \Cref{sec:infdivOdot} we deal with infinitely divisible moment sequences with respect to the Hadamard product.
In \Cref{sec:generators} we then give a full characterization of generators of diagonal positivity presevers.
We give a different approach to results from \cite{tyan75}.
In \Cref{sec:summary} we summarize our findings.

\section*{Acknowledgment}

We thank G.\ Blekherman valuable remarks and for the reference \cite{blekherman22arxiv}.
We thank C.\ Berg for providing us with the reference \cite{tyan75}.

\section{Preliminaries}
\label{sec:prelim}

\subsection{General Notation and Known Results}\label{subsec:genNot}

We use $\nset := \{1,2,3,\dots\}$, $\nset_0 = \{0,1,2,\dots\}$, and 
$\|x\|_2 := \sqrt{x_1^2 + \dots + x_n^2}$ for all $x=(x_1,\dots,x_n)^T\in\rset^n$ with $n\in\nset$.
$\chi_A$ denotes the characteristic function of a set $A\subseteq\rset^n$.
For two sets $A,B\subseteq\rset^n$ we define
\begin{equation}\label{eq:KmultL}
A\cdot B:=\{(a_1 b_1,\dots,a_n b_n) \,|\, (a_1,\dots,a_n)\in A,\ (b_1,\dots,b_n)\in B\}.
\end{equation}

We denote by $\delta_{i,j}$ the Kronecker delta, by $\delta_x$ the Dirac measure supported at $x\in\rset^n$, by $e_i := (\delta_{i,j})_{j=1}^n\in\rset^n$ the $i$-th standard unit vector of $\rset^n$, and by $\one := (1,\dots,1)^T\in\rset^n$.

For multi-indices $\alpha=(\alpha_1,\dots,\alpha_n)\in\nset_0^n$ we set
$\alpha! := \alpha_1!\cdots \alpha_n!$,
$|\alpha| := \alpha_1 + \dots + \alpha_n$, 
$x^\alpha := x_1^{\alpha_1}\cdots x_n^{\alpha_n}$,
$\partial^\alpha := \partial_1^{\alpha_1}\cdots\partial_n^{\alpha_n}$, and
$(x\partial)^{\alpha} := (x_1\partial_1)^{\alpha_1}\cdots (x_n\partial_n)^{\alpha_n}$.
For a real sequence $s = (s_\alpha)_{\alpha\in\nset_0^n}$ we define the Riesz functional $L_s$ as the linear functional $L_s:\rset[x_1,\dots,x_n]\to\rset$ by $L_s(x^\alpha) := s_\alpha$.

A symmetric matrix $A\in\rset^{n\times n}$ is called positive semi-definite $A\succeq 0$ if $x^T Ax\geq 0$ holds for all $x\in\rset^n$.
We have the following known results.

\begin{schurthm}[{\cite{schur11}}]\label{thm:schur}
Let $n\in\nset$ and $A=(a_{i,j})_{i,j=1}^n, B=(b_{i,j})_{i,j=1}^n\in\rset^{n\times n}$ be positive semi-definite matrices.
Then the Hadamard product is positive semi-definite, i.e., $A\circ B := (a_{i,j} b_{i,j})_{i,j=1}^n \succeq 0$.
\end{schurthm}

For linear maps $T:\rset[x_1,\dots,x_n]\to\rset[x_1,\dots,x_n]$ we have the following representation as differential operators.

\begin{lem}[folklore, see e.g.\ {\cite[Lem.\ 2.3]{netzer10}}]\label{lem:Tform}
Let $n\in\nset$.
Then the following are equivalent:
\begin{enumerate}[(i)]
\item $T:\rset[x_1,\dots,x_n]\to\rset[x_1,\dots,x_n]$ is linear.

\item There exist unique $q_\alpha\in\rset[x_1,\dots,x_n]$ such that
$\displaystyle T = \sum_{\alpha\in\nset_0^n} q_\alpha\cdot\partial^\alpha$.
\end{enumerate}
\end{lem}

\begin{rem}
Besides \Cref{lem:Tform} there is a second representation of linear maps $T$, namely
\begin{equation}\label{eq:Tkernel}
T = \sum_{i\in\nset_0} l_i(\,\cdot\,)\cdot p_i
\end{equation}
with linear $l_i:\rset[x_1,\dots,x_n]\to\rset$ and $p_i\in\rset[x_1,\dots,x_n]$ for all $i\in\nset_0$.
The representation (\ref{eq:Tkernel}) follows from a consequence of the \emph{Schwartz' Kernel Theorem} \cite[Ch.\ 50]{treves67} since $\rset[x_1,\dots,x_n]$ is nuclear \cite[p.\ 526, Cor.\ 2]{treves67}.
Hence, by \cite[eq.\ (50.18)]{treves67} we have for any nuclear space $E$ and Fr\'echet space $F$ that
\[L(E,F) \cong E'\widehat{\otimes} F\]
holds with the unique completion $E'\widehat{\otimes} F$ of $E'\otimes F$ since $E$ and hence also $E'$ are nuclear.
In fact, a representation (\ref{eq:Tkernel}) can easily be written down as
\[Tf = \sum_{\alpha\in\nset_0^n} l_\alpha(Tf)\cdot x^\alpha \qquad\text{for all}\ f\in\rset[x_1,\dots,x_n]\]
with $l_\alpha(x^\beta) := \delta_{\alpha,\beta}$ for all $\alpha,\beta\in\nset_0^n$.

If in (\ref{eq:Tkernel}) $l_i$ are $K$-moment functionals and $p_i\geq 0$ on $K\subseteq\rset^n$ closed then $T$ is a $K$-positivity preserver.
However, not every $K$-positivity preserver has a representation as (\ref{eq:Tkernel}).

To see this we take $K=\rset^n$ and $T = \id$ the identity.
Since $(x_1-a)^2$ with $a\in\rset$ are extreme rays of $\pos(\rset^n)$ they can not be written as a conic linear combination of any other $p\in\pos(\rset^n)$.
Hence, in (\ref{eq:Tkernel}) each $(x_1-a)^2$ with $a\in\rset$ must be contained in $\{p_i\}_{i\in\nset_0}$.
But since $\nset_0$ is countable and $\rset$ is uncountable this is a contradiction.

Therefore, in \cite{guterman08,borcea11,netzer10,didio24posPresConst,didio24posPres2arxiv} and also in the present work the representation in \Cref{lem:Tform} is used and not the representation (\ref{eq:Tkernel}).
\exmsymbol
\end{rem}

\subsection{The additive Convolution $*$}
\label{subsec:addConv}

\begin{dfn}[see e.g.\ {\cite[Sect.\ 3.9]{bogachevMeasureTheory}}]\label{dfn:addConv}
Let $n\in\nset$ and let $\mu$ and $\nu$ be $\sigma$-finite measures on $\rset^n$.
We define the \emph{additive convolution} $\mu*\nu$ by
\begin{multline*}
(\mu*\nu)(A) := \int_{\rset^n\times\rset^n} \chi_A(x+y)~\diff\mu(x)~\diff\nu(y)\\
= \int_{\rset^n} \mu(A-y)~\diff\nu(y) = \int_{\rset^n} \nu(A-x)~\diff\mu(x).
\end{multline*}
We define
\[\mu^{*0} := \delta_0 \quad\text{and}\quad \mu^{*k} := \underbrace{\mu*\dots *\mu}_{k\text{-times}}\]
for all $k\in\nset$.
\end{dfn}

With $a:\rset^n\times\rset^n\to\rset^n$, $(x,y)\mapsto x+y$ we have that \Cref{dfn:addConv} can be extended to (not necessarily $\sigma$-finite) measures by
\[\mu * \nu := (\mu\times\nu)\circ a^{-1}.\]

\begin{dfn}
Let $n,k\in\nset$.
We say that a measure $\mu$ on $\rset^n$ is \emph{$*$-divisible by $k$} if there exists a measure $\nu_k$ such that $\mu = \nu_k^{*k}$.
We say that $\mu$ is infinitely $*$-divisible if $\mu$ is $*$-divisible for all $k\in\nset$.
\end{dfn}

Infinitely $*$-divisible probability measures are fully characterized through the following form of their characteristic function.

\begin{lkthm}[\emph{additive version}, see e.g.\ {\cite[Thm.\ 8.1]{sato99}} {\cite[Cor.\ 15.8]{kallenberg02}}, or {\cite[Satz 16.17]{klenkewtheorie}}]\label{thm:levyKhinchin}
Let $n\in\nset$ and let $\mu$ be a probability measure.
Then the following are equivalent:
\begin{enumerate}[(i)]
\item $\mu$ is infinitely $*$-divisible.

\item There exists a vector $b\in\rset^n$, a symmetric matrix $\Sigma\in\rset^{n\times n}$ with $\Sigma\succeq 0$, and a $\sigma$-finite measure $\nu$ on $\rset^n$ with
\[\nu(\{0\}) = 0 \quad\text{and}\quad \int_{\rset^n} \min(\|x\|_2^2, 1)~\diff\nu(x)<\infty\]
such that
\[\log \int e^{itx}~\diff\mu(x) = itb - \frac{1}{2}t^T\Sigma t + \int e^{itx} - 1 - itx\cdot\chi_{\{\|x\|_2 < 1\}}(x)~\diff\nu(x)\]
holds for the characteristic function of $\mu$.
\end{enumerate}
\end{lkthm}

The measure $\nu$ in (ii) is called \emph{L\'evy measure}.
With the \Cref{thm:levyKhinchin} we previously fully characterized the generators of positivity preserving semi-groups on $\rset[x_1,\dots,x_n]$ in the constant coefficient case \cite[Main Thm.\ 4.11]{didio24posPresConst} and the non-constant coefficient case \cite[Thm.\ 5.12]{didio24posPres2arxiv}.

\begin{thm}[\emph{constant coefficient case} {\cite[Main Thm.\ 4.11]{didio24posPresConst}}]\label{thm:posGenConst}
Let $n\in\nset$ and let
\[A = \sum_{\alpha\in\nset_0^n\setminus\{0\}} \frac{a_\alpha}{\alpha!}\cdot\partial^\alpha \quad\in \rset[[\partial_1,\dots,\partial_n]].\]
Then the following are equivalent:
\begin{enumerate}[(i)]
\item $A$ is a generator of a positivity preserving semi-group $(e^{tA})_{t\geq 0}$.

\item There exists a symmetric matrix $\Sigma=(\sigma_{i,j})_{i,j=1}^n\in\rset^{n\times n}$, a vector $b=(b_1,\dots,b_n)^T\in\rset^n$, and a $\sigma$-finite measure $\nu$ on $\rset^n$ with
\[\nu(\{0\}) = 0 \quad\text{and}\quad \int_{\rset^n} |x^\alpha|~\diff\nu(x) < \infty\]
for all $\alpha\in\nset_0^n$ with $|\alpha|\geq 2$ such that
\begin{align*}
a_{e_i} &= b_i + \int_{\|x\|_2\geq 1} x_i~\diff\nu(x) && \text{for all}\ i=1,\dots,n,\\
a_{e_i+e_j} &= \sigma_{i,j} + \int_{\rset^n} x^{e_i+e_j}~\diff\nu(x) && \text{for all}\ i,j=1,\dots,n
\intertext{and}
a_\alpha &= \int_{\rset^n} x^\alpha~\diff\nu(x) && \text{for all}\ \alpha\in\nset_0^n\ \text{with}\ |\alpha|\geq 3.
\end{align*}
\end{enumerate}
\end{thm}

Note, that for an infinitely divisible measure $\mu$ with L\'evy measure $\nu$ we have that $\mu$ possesses all moments if and only if $\nu|_{\{\|x\|_2>1\}}$ possesses all moments, see \cite[Thm.\ 25.3]{sato99}.

\begin{thm}[\emph{non-constant coefficient case} {\cite[Thm.\ 5.12]{didio24posPres2arxiv}}]\label{thm:posGenNonConst}
Let $n\in\nset$ and let
\[A = \sum_{\alpha\in\nset_0^n\setminus\{0\}} \frac{a_\alpha}{\alpha!}\cdot\partial^\alpha\]
with $a_\alpha\in \rset[x_1,\dots,x_n]_{\leq |\alpha|}$ for all $\alpha\in\nset_0^n\setminus\{0\}$.
The following are equivalent:
\begin{enumerate}[(i)]
\item $A$ is the generator of a positivity preserving semi-group $(e^{tA})_{t\geq 0}$.

\item For every $y\in\rset^n$ there exist a symmetric matrix $\Sigma(y)=(\sigma_{i,j}(y))_{i,j=1}^n\in\rset^{n\times n}$, a vector $b(y)=(b_1(y),\dots,b_n(y))^T\in\rset^n$, and a $\sigma$-finite measure $\nu_y$ on $\rset^n$ with
\[\nu_y(\{0\}) = 0 \quad\text{and}\quad \int_{\rset^n} |x^\alpha|~\diff\nu_y(x) < \infty\]
for all $\alpha\in\nset_0^n$ with $|\alpha|\geq 2$ such that
\begin{align*}
a_{e_i}(y) &= b_i(y) + \int_{\|x\|_2\geq 1} x_i~\diff\nu_y(x) &&\text{for all}\ i=1,\dots,n,\\
a_{e_i+e_j}(y) &= \sigma_{i,j}(y) + \int_{\rset^n} x^{e_i+e_j}~\diff\nu_y(x) &&\text{for all}\ i,j=1,\dots,n
\intertext{and}
a_\alpha(y) &= \int_{\rset^n} x^\alpha~\diff\nu_y(x) && \text{for all}\ \alpha\in\nset_0^n\ \text{with}\ |\alpha|\geq 3.
\end{align*}
\end{enumerate}
\end{thm}

\Cref{thm:posGenNonConst} states that for a generator of a positivity preserving semi-group there exist Lévy triplets $(b(y), \Sigma(y), \nu_y)$ for every $y \in \rset^n$. 
This resembles \emph{Feller processes} which are multidimensional analogues of L\'evy processes allowing the L\'evy measure to depend on $y\in\rset^n$, see \cite[Cor.\ 2.23]{boettcher13}.

\subsection{The multiplicative Convolution $\odot$}
\label{subsec:multConv}

While the additive convolution of measures and hence (moment) sequences is textbook knowledge \cite[Sec.\ 3.9]{bogachevMeasureTheory} the multiplicative convolution is lesser known but known, see e.g.\ \cite{carcoma17} for the one dimensional formulation.
We give here the general multivariate formulations.

\begin{dfn}\label{dfn:odotMeasures}
Let $n\in\nset$ and let
\[m:\rset^n\times\rset^n\to\rset^n,\quad (x=(x_1,\dots,x_n),y=(y_1,\dots,y_n)) \mapsto (x_1 y_1,\dots,x_n y_n).\]
For measures $\mu$ and $\nu$ we define the \emph{multiplicative convolution $\mu\odot\nu$} by
\[\mu\odot\nu := (\mu\times\nu)\circ m^{-1}.\]
We define
\[\mu^{\odot 0} := \delta_1 \quad\text{and}\quad \mu^{\odot k} := \underbrace{\mu\odot\dots\odot\mu}_{k\text{-times}}\]
for all $k\in\nset$.
\end{dfn}

Alternatively, if $\mu$ and $\nu$ are $\sigma$-finite we can define $\mu\odot\nu$ by
\[(\mu\odot\nu)(A):= \int_{\rset^n\times\rset^n} \chi_A(x_1 y_1,\dots,x_n y_n)~\diff\mu(x_1,\dots,x_n)~\diff\nu(y_1,\dots,y_n)\]
for all Borel sets $A\subseteq\rset^n$.
$\odot$ has the following (well-known/simple) properties.

\begin{lem}\label{lem:odotProps}
Let $n\in\nset$, let $\mu$, $\nu$, and $\omega$ be measures on $\rset^n$, let $A$ be a Borel set of $\rset^n$, let $u=(u_1,\dots,u_n)$, $v=(v_1,\dots,v_n)\in\rset^n$, and let $a,b\geq 0$ and $c>0$ be constants.
Then the following hold:
\begin{enumerate}[(i)]
\item $\mu\odot\nu = \nu\odot\mu$.

\item $(a\mu + b\nu)\odot\omega = a\cdot \mu\odot\omega + b\cdot\nu\odot\omega$.

\item $(\mu\odot\nu)\odot\omega = \mu\odot (\nu\odot\omega)$.

\item $(\mu\odot\nu)(\rset^n) = \mu(\rset^n)\cdot\nu(\rset^n)$.

\item $\delta_0\odot\mu = \mu(\rset^n)\cdot\delta_0$.\label{item:first}

\item $\delta_\one\odot\mu = \mu$ with $\one := (1,\dots,1)\in\rset^n$.\label{item:middle}

\item $(\delta_{c\cdot\one}\odot\mu)(A) = \mu(c^{-1}\cdot A)$.\label{item:last}

\item $\delta_u\odot\delta_v = \delta_{u\odot v}$ with $u\odot v := (u_1 v_1,\dots, u_n v_n)$.\label{item:last2}
\end{enumerate}
\end{lem}

\begin{lem}\label{lem:odot_inf_diff_equals_log_inf_diff}
Let $\mu$ be a probability measure on $(0, \infty)^n$. 
Then the following are equivalent:
\begin{enumerate}[(i)]
\item $\mu$ is infinitely $\odot$-divisible.
\item $\mu \circ \ln^{-1}$ is infinitely $*$-divisible with $\ln x = (\ln x_1, \dots, \ln x_n)$, $x \in (0, \infty)^n$.
\end{enumerate}
\end{lem}

As a corollary of \Cref{lem:odot_inf_diff_equals_log_inf_diff} we obtain a version of the Lévy-Khintchine Theorem for infinitely $\odot$-divisible measures on $(0,\infty)^n$.
Alternatively this representation can be derived using Hunt's theorem, see \cite[§6.3]{hunt56} or \cite[Thm.\ 5.5.1]{applebaum14}.

\begin{lkthm}[multiplicative version on $(0,\infty)^n$]
Let $n \in \nset$ and let $\mu$ be a probability measure on $(0, \infty)^n$. 
Then the following are equivalent:
\begin{enumerate}[(i)]
\item $\mu$ is infinitely $\odot$-divisible.
\item There exist a vector $b \in \rset^n$, a symmetric matrix $\Sigma \in \rset^{n \times n}$ with $\Sigma \succeq 0$, and a $\sigma$-finite measure $\nu$ on $(0, \infty)^n$ with
\[\nu(\{\one\}) = 0\quad\text{and}\quad \int_{(0, \infty)^n} \min(\|\ln x\|_2^2, 1) ~\diff\nu(x) < \infty\]
with $\ln(x_1, \dots, x_n) = (\ln(x_1), \dots, \ln(x_n)) \in \rset^n$ for $x \in (0, \infty)^n$ such that
\begin{multline*}
	\log \int_{(0, \infty)^n} e^{i\skl{t}{\ln x}}~\diff\mu(x) 
	= i\skl{b}{t} - \frac{1}{2}t^T\Sigma t 
\\	+ \int_{(0, \infty)^n} e^{i\skl{t}{\ln x}} - 1 - i\cdot\skl{t}{\ln x}\cdot\chi_{\{\|\ln x\|_2 < 1\}}(x) ~\diff\nu(x)
\end{multline*}
holds for all $t \in (0, \infty)^n$.
\end{enumerate}
\end{lkthm}

With $e^{i\skl{t}{\ln x}} = x^{it}$ (understood as a multivariate power) this representation is similar to the one used by Berg \cite{berg04,berg05,berg07} and Tyan \cite{tyan75} to characterize infinitely $\odot$-divisible moment sequences.
This theorem instead characterizes probability measures, meaning the moments of these measures do not need to exist.
For the diagonal operator case one would want to use this L\'evy-Khinchine representation and the technique used in the proof of \Cref{thm:posGenConst}.
But this requires additional assumptions such that the coefficients of the generator exist.
With \Cref{thm:diagonalGenerators} these restrictions can be dropped.

\section{Characterization of all linear Maps $S:\rset^{\nset_0^n}\to\rset^{\nset_0^n}$ preserving $K$-Moment Sequences}
\label{sec:charLinMaps}

Linear maps preserving moment sequences have been characterized before \cite{carcoma17}.
But these are $\rset^n$-moment sequences.
A characterization of $K$-moment sequence preserving linear maps for general closed $K\subseteq\rset^n$ with $n\in\nset$ was open so far.
We use the results from \cite{didio24posPres2arxiv} to solve this general characterization.

Recall, that the set of sequences
\[\rset^{\nset_0^n} := \{s = (s_\alpha)_{\alpha\in\nset_0^n} \,|\, s_\alpha\in\rset\ \text{for all}\ \alpha\in\nset_0^n\}\]
is a \emph{Fr\'echet space} when equipped with the family of semi-norms
\[|s|_d := \max_{\alpha\in\nset_0^n: |\alpha|\leq d} |s_\alpha|,\]
i.e., it is metrizable, complete, and locally convex.
Completeness in $\rset^{\nset_0^n}$ is simply the following.
A sequences $s^{(i)} = (s_\alpha^{(i)})_{\alpha\in\nset_0^n}$ converges to a sequence $s$, i.e., $s^{(i)}\xrightarrow{i\to\infty} s$, if and only if $s_\alpha^{(i)}\xrightarrow{i\to\infty} s_\alpha$ for all $\alpha\in\nset_0^n$.
See \cite[pp.\ 91--92, Exm.\ III]{treves67} since $\rset^{\nset_0^n} \cong \rset[[x_1,\dots,x_n]]$.

The set of sequences $\rset^{\nset_0^n}$ has a topological dual.
It is the set of polynomials $\rset[x_1,\dots,x_n]$ equipped with the LF-topology, see \cite[Ch.\ 22]{treves67}.
Note, any LF-space is complete \cite[Thm.\ 13.1]{treves67} and hence the sequence of polynomials $p^{(i)} = \sum_{\alpha\in\nset_0^n: |\alpha|\leq d_i} c^{(i)}_\alpha\cdot x^\alpha$, $d_i\in\nset_0$, converges to a polynomial $p = \sum_{\alpha\in\nset_0^n: |\alpha|\leq d} c_\alpha\cdot x^\alpha$ in the LF-topology, i.e., $p^{(i)}\xrightarrow{i\to\infty} p$, if and only if $\sup_{i\in\nset} d_i = d < \infty$ and $c^{(i)}_\alpha \xrightarrow{i\to\infty} c_\alpha$.

The duality between $\rset^{\nset_0^n}$ and $\rset[x_1,\dots,x_n]$ can be expressed in a natural way by
\[
	\langle s,p\rangle := \sum_{\alpha\in\nset_0^n} s_\alpha c_\alpha
\]
and $\langle s,\,\cdot\,\rangle:\rset[x_1,\dots,x_n]\to\rset$ is the well-known Riesz functional $L_s = \langle s,\,\cdot\,\rangle$.
For more on topological vector spaces see e.g.\ \cite{treves67,koethe69TopVecSp1,koethe79TopVecSp2,schaef99}.

Since every linear operator $T:\rset[x_1,\dots,x_n]\to\rset[x_1,\dots,x_n]$ is given by \Cref{lem:Tform} with unique $q_\alpha\in\rset[x_1,\dots,x_n]$ we have that the dual operator
\[T^*:\rset^{\nset_0^n}\to\rset^{\nset_0^n}\]
defined by
\[\langle T^* s,p\rangle := \langle s, Tp\rangle\]
for all $s\in\rset^{\nset_0^n}$ and $p\in\rset[x_1,\dots,x_n]$ is given by
\begin{equation}\label{eq:TstarForm}
T^* = \sum_{\alpha\in\nset_0^n} (-1)^{|\alpha|}\cdot \partial^\alpha M_{q_\alpha}
\end{equation}
where $M_{q_\alpha}$ is the shift operator given by $M_{x^\beta}s := (s_{\alpha+\beta})_{\alpha\in\nset_0^n}$ and $M_{af+g} = aM_f + M_g$ for all $a\in\rset$ and $f,g\in\rset[x_1,\dots,x_n]$.
The $\partial^\alpha$ on $\rset^{\nset_0^n}$ are defined in the distributional sense as in \cite{didio23gaussian}, see e.g.\ \cite{grubbDistributions} for more on distributions.
Hence, every linear operator $S:\rset^{\nset_0^n}\to\rset^{\nset_0^n}$ is given by
\begin{equation}
S = \sum_{\alpha\in\nset_0^n} (-1)^{|\alpha|}\cdot \partial^\alpha M_{q_\alpha}
\end{equation}
with unique $q_\alpha\in\rset[x_1,\dots,x_n]$ since $T^{**} = T$ we have $S^*:\rset[x_1,\dots,x_n]\to\rset[x_1,\dots,x_n]$ linear.
Additionally, $(Ss)_\alpha = \langle Ss, x^\alpha\rangle = \langle s, S^* p\rangle$ with $\alpha\in\nset_0^n$ shows that each coordinate of $Ss$ is a finite linear combination of the coordinates $s_\alpha$ of $s$.
We summarize that in the following.

\begin{lem}
Let $n\in\nset$.
Then the following are equivalent:
\begin{enumerate}[(i)]
\item $S:\rset^{\nset_0^n}\to\rset^{\nset_0^n}$ is linear.

\item $S = \sum_{\alpha\in\nset_0^n} (-1)^{|\alpha|}\cdot \partial^\alpha M_{q_\alpha}$ with unique $q_\alpha\in\rset[x_1,\dots,x_n]$ for all $\alpha\in\nset_0^n$.
\end{enumerate}
The coordinates $(Ss)_\alpha$ of $Ss$ are a finite linear combination of coordinates of $s$.
\end{lem}

With these remarks and \cite[Thm.\ 4.5]{didio24posPres2arxiv} we get the following full characterization of linear operators $S:\rset^{\nset_0^n}\to\rset^{\nset_0^n}$ which preserve $K$-moment sequences.

\begin{thm}\label{thm:linS}
Let $n\in\nset$, let $K\subseteq\rset^n$ be closed, and let $S:\rset^{\nset_0^n}\to\rset^{\nset_0^n}$ be linear, i.e.,
\[S = \sum_{\alpha\in\nset_0^n} (-1)^{|\alpha|}\cdot\partial^\alpha M_{q_\alpha}\]
for some $q_\alpha\in\rset[x_1,\dots,x_n]$.
Then the following are equivalent:
\begin{enumerate}[(i)]
\item The linear operator $S$ preserves $K$-moment sequences.

\item For all $y\in K$ we have that $(\alpha!\cdot q_\alpha(y))_{\alpha\in\nset_0^n}$ is a $(K-y)$-moment sequence.
\end{enumerate}
\end{thm}
\begin{proof}
By
\[\langle Ss,p\rangle = \langle s, S^* p\rangle\]
for all $s\in\rset^{\nset_0^n}$ and $p\in\rset[x_1,\dots,x_n]$ we have that $S$ preserves $K$-moment sequences if and only if $S^*:\rset[x_1,\dots,x_n]\to\rset[x_1,\dots,x_n]$ preserves $K$-non-negative polynomials.
Linear maps $\rset[x_1,\dots,x_n]\to\rset[x_1,\dots,x_n]$ preserving $K$-non-negative polynomials have been fully characterized by \cite[Thm.\ 4.5]{didio24posPres2arxiv} which gives the equivalence to (ii).
\end{proof}

From the previous we deduce the following results.
In \Cref{thm:linS} we already proved the following.

\begin{cor}\label{cor:eqST}
Let $n\in\nset$, let $K\subseteq\rset^n$ be closed, and let $S:\rset^{\nset_0^n}\to\rset^{\nset_0^n}$ be linear.
$S$ preservers $K$-moment sequences iff $S^*:\rset[x_1,\dots,x_n]\to\rset[x_1,\dots,x_n]$ preserves $K$-positive polynomials.
\end{cor}

\begin{cor}
Let $n\in\nset$, let $K\subseteq\rset^n$ be closed, and let $S:\rset^{\nset_0^n}\to\rset^{\nset_0^n}$, $s\mapsto Ss := M_q s$ be with $q\in\rset[x_1,\dots,x_n]$.
Then the following are equivalent:
\begin{enumerate}[(i)]
\item The operator $S$ preserves $K$-moment sequences.

\item $q\in\pos(K)$.
\end{enumerate}
\end{cor}
\begin{proof}
By \Cref{thm:linS} we have that $S$ preserves $K$-moment sequences if and only if $(\alpha!\cdot q_\alpha(y))_{\alpha\in\nset_0^n}$ with $q_0 = q$ and $q_\alpha = 0$ otherwise is a $(K-y)$-moment sequence for all $y\in K$.
Since $(\alpha!\cdot q_\alpha(y))_{\alpha\in\nset_0^n}$ is represented by $q(y)\cdot\delta_0$ we have that (i) holds if and only if (ii) holds.
\end{proof}

\section{Diagonal Operators $T:\rset[x_1,\dots,x_n]\to\rset[x_1,\dots,x_n]$}
\label{sec:diagOp}

We now turn to our main topic: diagonal operators.
In $Tx^\alpha = t_\alpha x^\alpha$ for all $\alpha\in\nset_0^n$ the sequences $t = (t_\alpha)_{\alpha\in\nset_0^n}$ is called the \emph{diagonal sequence} of $T$.
The following is known.
We give a shorter proof as in \cite{borcea11}.

\begin{thm}[{\cite[Cor.\ 4.3]{borcea11}}]\label{thm:borceaDiag}
Let $n\in\nset$ and let $T:\rset[x_1,\dots,x_n]\to\rset[x_1,\dots,x_n]$ be a diagonal operator with diagonal sequence $t = (t_\alpha)_{\alpha\in\nset_0^n}$.
Then the following are equivalent:
\begin{enumerate}[(i)]
\item $T\pos(\rset^n)\subseteq\pos(\rset^n)$.

\item $t$ is a moment sequences.
\end{enumerate}
If $\mu$ is a representing measure of $t$ then $T$ fulfills
\begin{equation}\label{eq:Tmu}
(Tp)(x) = \int_{\rset^n} p(x_1 y_1,\dots,x_n y_n)~\diff\mu(y_1,\dots,y_n)
\end{equation}
for all $p\in\rset[x_1,\dots,x_n]$.
\end{thm}
\begin{proof}
While (ii) $\Rightarrow$ (i) is clear from (\ref{eq:Tmu}) the reverse follows from $L(p) := (Tp)(1,\dots,1)\geq 0$ for all $p\in\pos(\rset^n)$ and $L(x^\alpha) = t_\alpha$ for all $\alpha\in\nset_0^n$ with Haviland's Theorem \cite{havila36}.
\end{proof}

\begin{rem}\label{rem:diagOpCoefficients}
Diagonal operators $T:\rset[x_1,\dots,x_n]\to\rset[x_1,\dots,x_n]$ have three representations:
\begin{enumerate}[\;\;\itshape a)]
\item $Tx^\alpha = t_\alpha x_\alpha$,

\item $\displaystyle T = \sum_{\alpha\in\nset_0^n} \frac{c_\alpha}{\alpha!}\cdot x^\alpha\cdot\partial^\alpha$, and

\item $\displaystyle T = \sum_{\alpha\in\nset_0^n} \frac{d_\alpha}{\alpha!}\cdot (x\partial)^\alpha$.
\end{enumerate}
Any of these sequences $(t_\alpha)_{\alpha\in\nset_0^n}$, $(c_\alpha)_{\alpha\in\nset_0^n}$, or $(d_\alpha)_{\alpha\in\nset_0^n}$ determine $T$ uniquely.
We switch between these representations, mostly between (a) and (b).
We therefore give here the formulas transforming one into the others:
\begin{alignat*}{2}
t_\alpha \quad &=\quad \sum_{\beta\preceq\alpha} \binom{\alpha}{\beta}\cdot c_\beta& \quad&=\quad \sum_{\beta\in\nset_0^n} \frac{\alpha^\beta}{\beta!}\cdot d_\beta,\\
c_\alpha \quad &=\quad \sum_{\beta\preceq\alpha} (-1)^{|\alpha-\beta|}\cdot\binom{\alpha}{\beta}\cdot t_\beta &\quad &=\quad \sum_{\beta\succeq\alpha} \frac{\alpha!}{\beta!}\cdot \stirlingtwo{\alpha}{\beta}\cdot d_\beta\\
d_\alpha \quad &=\quad \sum_{\beta \succeq \alpha} \sum_{\gamma \preceq \beta} (-1)^{|\alpha+\gamma|} \frac{\alpha!}{\beta!} \cdot \stirlingone{\beta}{\alpha} \cdot \binom{\beta}{\gamma} \cdot t_\gamma &\quad&=\quad \sum_{\beta\succeq\alpha} (-1)^{|\beta-\alpha|} \frac{\alpha!}{\beta!}\cdot \stirlingone{\beta}{\alpha}\cdot c_\beta
\end{alignat*}
These relations follow by direct computations, e.g.\ with the relations
\[x^\alpha\partial^\alpha = \sum_{\beta \preceq \alpha} (-1)^{|\alpha-\beta|}\cdot \begin{bmatrix}\alpha\\ \beta\end{bmatrix}\cdot (x\partial)^\beta
\qquad\text{and}\qquad
(x\partial)^\alpha = \sum_{\beta \preceq \alpha} \begin{Bmatrix}\alpha\\ \beta\end{Bmatrix}\cdot x^\beta\cdot \partial^\beta\]
where $\begin{bmatrix}\alpha\\ \beta\end{bmatrix}$ are the (\emph{unsigned}) \emph{Stirling numbers of the first kind} and $\begin{Bmatrix}\alpha\\ \beta\end{Bmatrix}$ are the \emph{Stirling numbers of the second kind}, see e.g.\ \cite{mansour16}.
\exmsymbol
\end{rem}

It is clear that if $(c_\alpha)_{\alpha\in\nset_0^n}$ has only finitely many non-zero entries then only finitely many entries in $(t_\alpha)_{\alpha\in\nset_0^n}$ are zero.
Also the reverse can hold, i.e., finitely many entries in $(t_\alpha)_{\alpha\in\nset_0^n}$ are non-zero but only finitely many entries in $(c_\alpha)_{\alpha\in\nset_0^n}$ are zero.

\begin{exm}\label{exm:T0}
Define the evaluation operator
\[T_0:\rset[x_1,\dots,x_n]\to\rset[x_1,\dots,x_n],\quad p\mapsto T_0 p = p(0),\]
i.e., $T_0 x^\alpha = \delta_{\alpha,0} x^\alpha$ for all $\alpha\in\nset_0^n$.
Then $T_0 = \sum_{\alpha\in\nset_0^n} \frac{(-1)^{|\alpha|}}{\alpha!}\cdot x^\alpha\cdot\partial^\alpha$ holds.
\exmsymbol
\end{exm}

\begin{exm}
For the identity $\id:\rset[x_1,\dots,x_n]\to\rset[x_1,\dots,x_n]$ defined by $\id\,p := p$ we have $\id = 1 = \sum_{\alpha\in\nset_0^n} \frac{\delta_{\alpha,0}}{\alpha!}\cdot x^\alpha\cdot\partial^\alpha$.
\exmsymbol
\end{exm}

Every diagonal operator $T$ has a representation (a) and (b).
But it does not necessarily have a representation (c) as the next example shows.

\begin{exm}[\Cref{exm:T0} continued]\label{exm:dalphaNonEx}
The evaluation operator
\[T_0 = \sum_{\alpha\in\nset_0^n} \frac{(-1)^{|\alpha|}}{\alpha!}\cdot x^\alpha\cdot\partial^\alpha\]
from \Cref{exm:T0} has no type (c) representation since the sum
\[d_\alpha = \sum_{\beta\succeq\alpha} (-1)^{|\alpha|}\cdot \frac{\alpha!}{\beta!}\cdot\begin{bmatrix}\beta\\ \alpha\end{bmatrix}\]
diverges for all $\alpha\neq 0$.
\exmsymbol
\end{exm}

In \Cref{rem:diagOpCoefficients} we have seen that if $t = (t_\alpha)_{\alpha\in\nset_0^n}$ is represented by $\mu_t$ then $c = (c_\alpha)_{\alpha\in\nset_0^n}$ is represented by $\mu_c := \mu_t (\,\cdot\,+\one)$ and hence $\mu_t = \mu_c(\,\cdot\,-\one)$, see \cite[p.\ 79]{borcea11}.
A similar statement holds for the coefficients $d_\alpha$ in the type (c) representation.

\begin{prop}
Let $n\in\nset$ and let $T: \rset[x_1, \dots, x_n] \to \rset[x_1, \dots, x_n]$ be a diagonal operator.
Then the following holds:
\begin{enumerate}[(i)]
\item If there exists a representing measure $\mu_t$ on $(0, \infty)^n$ of the diagonal sequence $t = (t_\alpha)_{\alpha \in \nset_0^n}$ with
\[
	d_\alpha := \int_{(0, \infty)^n} (\ln x)^\alpha ~\diff\mu_t(x) \in (-\infty, \infty)
\]
for all $\alpha \in \nset_0^n$ where $\ln x = (\ln x_1, \dots, \ln x_n)$ then 
\[
	T = \sum_{\alpha \in \nset_0^n} \frac{d_\alpha}{\alpha!}\cdot (x\partial)^\alpha.
\]

\item If there exists a representing measure $\mu_d$ on $\rset^n$ of the $(x\partial)^\alpha$-coefficient sequence $d = (d_\alpha)_{\alpha \in \nset_0^n}$ with
\[
	t_\alpha := \int_{\rset^n} e^{\skl{\alpha}{x}} ~\diff\mu_d(x) \in (-\infty, \infty)
\]
for all $\alpha \in \nset_0^n$ then 
\[
	Tx^\alpha = t_\alpha x^\alpha
\]
for all $\alpha \in \nset_0^n$.
\end{enumerate}
\end{prop}
\begin{proof}
(i): 
Looking at the diagonal values we obtain: 
\begin{align*}
	t_\beta 
	&= \int_{(0, \infty)^n} x^\beta ~\diff\mu_t(x)
\\	&= \int_{(0, \infty)^n} e^{\skl{\ln x}{\beta}} ~\diff\mu_t(x)
\\	&= \int_{(0, \infty)^n} \sum_{\alpha \in \nset_0^n} \frac{\beta^\alpha}{\alpha!} \cdot (\ln x)^\alpha ~\diff\mu_t(x)
\\	&= \sum_{\alpha \in \nset_0^n} \frac{\beta^\alpha}{\alpha!}\cdot d_\alpha
\end{align*}
for all $\beta \in \nset_0^n$. 
With \Cref{rem:diagOpCoefficients} it now follows that $T = \sum_{\alpha \in \nset_0^n} \frac{d_\alpha}{\alpha!} (x\partial)^\alpha$. 
This trick only works on $(0, \infty)^n$ as the logarithm $\ln x$ is not defined otherwise.
\par (ii):
Using \Cref{rem:diagOpCoefficients} the diagonal values of the operator corresponding to $d$ are
\[\sum_{\alpha \in \nset_0^n} \frac{\beta^\alpha}{\alpha!} d_\alpha = \int_{\rset^n} \sum_{\alpha \in \nset_0^n} \frac{\beta^\alpha}{\alpha!} x^\alpha ~\diff\mu_d(x)
= \int_{\rset^n} e^{\skl{x}{\beta}} ~\diff\mu_d(x)
= t_\beta\]
for all $\beta \in \nset_0^n$.
Hence, $Tx^\beta = t_\beta x^\beta$ for all $\beta \in \nset_0^n$.
\end{proof}

\section{The Product of Diagonal Positivity Preservers}
\label{sec:prodDiag}

The product of moment sequences has been investigated before.
We already mentioned the work \cite{blekherman22arxiv} where the multidimensional case is treated.
Other works treat only the univariate case which is easily solved by the \Cref{thm:schur}.
In the multidimensional case the \Cref{thm:schur} is not even applicable but further technical approximations of the moment sequence are required.
We show here that the \Cref{thm:schur} is not necessary.
We use only the diagonal operators and show that with these the problem becomes trivial.
We define the following.

\begin{dfn}\label{dfn:odot}
Let $k,n\in\nset$ and let $s = (s_\alpha)_{\alpha\in\nset_0^n}$ and $t=(t_\alpha)_{\alpha\in\nset_0^n}\in\rset^{\nset_0^n}$ be real sequences.
We define
\[s\odot t := (s_\alpha\cdot t_\alpha)_{\alpha\in\nset_0^n},\quad s^{\odot 0} := (1)_{\alpha\in\nset_0^n},\quad \text{and}\quad s^{\odot k} := \underbrace{s\odot\dots\odot s}_{k\text{-times}} = (s_\alpha^k)_{\alpha\in\nset_0^n}.\]
\end{dfn}

This is the \emph{Hadamard product} of (moment) sequences.
\Cref{thm:borceaDiag} immediately implies the following.

\begin{thm}\label{thm:product}
Let $n\in\nset$ and let $s,t\in\rset^{\nset_0^n}$ be two moment sequences with representing measures $\mu$ and $\nu$.
Then $s\odot t$ is a moment sequence with representing measure $\mu\odot\nu$ which fulfills $\supp\mu\odot\nu = \supp\mu \cdot \supp\nu$.
\end{thm}
\begin{proof}
Follows immediately from Definitions \ref{dfn:odotMeasures} and \ref{dfn:odot} and \Cref{thm:borceaDiag}.
\end{proof}

We have the following easy consequence of the previous theorem.
Recall, that a moment sequence is called \emph{indeterminate} if it has more than one representing measure.

\begin{cor}
Let $n\in\nset$ and let $s,t\in\rset^{\nset_0^n}$ be two moment sequences.
If $s$ is indeterminate then $s\odot t$ is indeterminate.
\end{cor}

\begin{rem}\label{rem:product}
\Cref{thm:product} not only states that the map $f:\rset\to\rset$, $x\mapsto x^2$ induces a moment sequence preserving map $f$.
By induction any map $f(x) = \sum_{k=0}^d a_k\cdot x^k$ with $d\in\nset_0$ and $a_k\geq 0$ for all $k=0,\dots,d$ preserves moment sequences entry-wise.
If $s$ is represented by $\mu$ then $f(\odot s) = \sum_{k=0}^d a_k\cdot s^{\odot k}$ is represented by the measure $f(\odot\mu):= \sum_{k=0}^d a_k\cdot \mu^{\odot k}$.
Under appropriate conditions on the coefficients $a_k$ the limit $d\to\infty$ exists for the sequence $f(\odot s)$ and the (signed) measure $f(\odot\mu)$.
That $f$ is entire is sufficient.
\exmsymbol
\end{rem}

The following is a generalization of \cite[Thm.\ 3.1 and 9.1]{belton22} with much shorter proofs.
Only one direction from \cite[Thm.\ 3.1]{belton22} must be used.
We did not find a way to simplify the original proof in \cite[Thm.\ 3.1]{belton22} (yet).

\begin{cor}\label{cor}
Let $f:\rset\to\rset$.
Then the following are equivalent:
\begin{enumerate}[(i)]
\item $f$ maps $[-1,1]$-moment sequences to $[-1,1]$-moment sequence.

\item $f$ is an absolutely monotonic entire function.

\item $f$ is entire and $f$ maps $\rset$-moment sequence to $\rset$-moment sequences.

\item $f$ is entire and $\big(f(2^{-k})\big)_{k\in\nset_0}$ is a moment sequence.

\item $f$ maps $[-1,1]^n$-moment sequences to $[-1,1]^n$-moment sequences for a fixed $n\in\nset$.

\item $f$ maps $[-1,1]^n$-moment sequences to $[-1,1]^n$-moment sequences for all $n\in\nset$.

\item $f$ is entire and $f$ maps $\rset^n$-moment sequences to $\rset^n$-moment sequences for a fixed $n\in\nset$.

\item $f$ is entire and $f$ maps $\rset^n$-moment sequences to $\rset^n$-moment sequences for all $n\in\nset$.
\end{enumerate}
\end{cor}
\begin{proof}
(i) $\Rightarrow$ (ii): That is the ``only if'' direction in \cite[Thm.\ 3.1]{belton22}.

(ii) $\Rightarrow$ (i):
Let $s$ be a $[-1,1]$-moment sequence with representing measure $\mu$.
Then $f(\odot s)$ is represented by $f(\odot\mu) = \sum_{k\in\nset_0} a_k\cdot \mu^{\odot k}$ and
\begin{equation}\label{eq:support-11}
\supp f(\odot\mu) = \bigcup_{k\in\nset_0} \supp \mu^{\odot k} \subseteq \bigcup_{k\in\nset} \underbrace{[-1,1]\cdot{\dots}\cdot[-1,1]}_{k\text{-times, see (\ref{eq:KmultL})}} = [-1,1]
\end{equation}
by \Cref{thm:product}.
That $f(\odot\mu)$ represents $f(\odot s)$, i.e., the limit $d\to\infty$ in \Cref{rem:product} exists, follows from (\ref{eq:support-11}) and
\begin{equation}\label{eq:fsuppR}
(f(\odot\mu))(\rset) = \sum_{k\in\nset_0} a_k\cdot (\mu^{\odot k})(\rset) = \sum_{k\in\nset_0} a_k\cdot (\mu(\rset))^k = f(\mu(\rset))<\infty.
\end{equation}

(ii) $\Rightarrow$ (iii): 
Eq.\ (\ref{eq:fsuppR}) does not require $\supp\mu\subseteq [-1,1]$ but holds for arbitrary $\supp\mu\subseteq\rset$ since $f$ is entire.

(iii) $\Rightarrow$ (iv): Clear since $(2^{-k})_{k\in\nset_0}$ is a moment sequence with representing measure $\delta_{1/2}$.

(iv) $\Rightarrow$ (ii):
The sequence $(2^{-k})_{k\in\nset_0}$ is a moment sequence with (unique) representing measure $\delta_{1/2}$.
Hence, by \Cref{thm:product} and \Cref{lem:odotProps} (\ref{item:last2}) we have that $(f(2^{-k}))_{k\in\nset_0}$ is represented by the (possibly signed) measure
\[f(\odot\delta_{2^{-1}}) = \sum_{k\in\nset_0} a_k\cdot\delta_{2^{-k}}\]
with $\supp f(\odot\delta_{2^{-1}}) \subseteq \{2^{-k} \,|\, k\in\nset_0\}\subseteq [0,1]$ and $|f(\odot\mu)|(\rset) \leq \sum_{k\in\nset_0} |a_k|\cdot (\mu(\rset))^k < \infty$.
Assume there exists a $k\in\nset_0$ such that $a_k < 0$.
Since $\rset[x]$ is dense in $C([0,1],\rset)$ there exists a $p\in\rset[x]$ such that
\[p\geq 0\ \text{on}\ [0,1],\quad p(2^{-k}) = 1,\quad \text{and}\quad p(2^{-m})\leq \frac{|a_k|}{2\sum_{l\in\nset_0\setminus\{k\}} |a_l|}\]
for all $m\in\nset_0\setminus\{k\}$.
Then with the Riesz functional $L_{f(\odot s)}:\rset[x]\to\rset$ defined by $L_{f(\odot s)}(x^m) = f(s_m)$ for all $m\in\nset_0$ we have
\[L_{f(\odot s)}(p) = \sum_{l\in\nset_0} a_l\cdot p(2^{-l})
\leq a_k + \sum_{m\in\nset_0\setminus\{k\}} |a_m|\cdot p(2^{-m})
\leq a_k + \frac{|a_k|}{2} < 0,\]
i.e., $f(\odot s)$ is not a moment sequence.
Hence, $a_k\geq 0$ holds for all $k\in\nset_0$.

(ii) $\Rightarrow$ (v), \dots, (viii): 
Follows with the same arguments from (ii) $\Rightarrow$ (i) for general $n\in\nset$ from \Cref{thm:product} and \Cref{lem:odotProps}.

(v) or (vi) $\Rightarrow$ (i):
For the multi-moment sequence $s = (s_\alpha)_{\alpha\in\nset_0^n}$ take the marginal moment sequence $s^{[1]} := (s_{(k,0,\dots,0)})_{k\in\nset_0}$.

(vii) or (viii) $\Rightarrow$ (iii):
For the multi-moment sequence $s = (s_\alpha)_{\alpha\in\nset_0^n}$ take the marginal moment sequence $s^{[1]} := (s_{(k,0,\dots,0)})_{k\in\nset_0}$.
\end{proof}

We see that positivity preservers are a powerful tool, especially the diagonal operators.

\section{A Proof of Schur's Theorem with Diagonal Positivity Preservers}
\label{sec:schur}

We have seen in the previous section that the \Cref{thm:schur} can be completely removed from the proof about the Hadamard product of (moment) sequences by employing diagonal operators.
We now go even a step further.
We show that the \Cref{thm:schur} can easily be proved with diagonal operators, i.e., we add to the existing list of proofs of the \Cref{thm:schur} a new proof.
We need the following homogeneous version of \Cref{thm:borceaDiag}.

\begin{prop}\label{prop:borceaDiagHom}
Let $n,d\in\nset_0$, let
\[\cP(n,2d) := \{ f\in\rset[x_0,\dots,x_n]_{=2d} \,|\, f(x)\geq 0\ \text{for all}\ x\in\rset^{n+1}\}\]
be the set of all non-negative homogeneous polynomials of degree $2d$ in $n$ variables, let $t = (t_\alpha)_{\alpha\in\nset_0^n:|\alpha|=2d}$ be a real sequences, and let
\[T:\rset[x_0,\dots,x_n]_{=2d}\to\rset[x_0,\dots,x_n]_{=2d}\]
be linear with $Tx^\alpha = t_\alpha x^\alpha$ for all $\alpha\in\nset_0^n$ with $|\alpha|=2d$.
Then the following are equivalent:
\begin{enumerate}[(i)]
\item $T\cP(n,2d)\subseteq\cP(n,2d)$.

\item $L_t\in \cP(n,2d)^*$.
\end{enumerate}
\end{prop}
\begin{proof}
(i) $\Rightarrow$ (ii):
We have $L_t(p) = (Tp)(\one) \geq 0$ for all $p\in\cP(n,2d)$.

(ii) $\Rightarrow$ (i):
Let $x\in\pset\rset^n$.
Since $\pset\rset^n$ is invariant under rotation and also $\cP(n,2d)$ is invariant under this coordinate change we can assume without loss of generality $x=\one$.
Hence, $(Tp)(\one) = L_t(p)\geq 0$ for all $p\in\cP(n,2d)$ and therefore $T\cP(n,2d)\subseteq\cP(n,2d)$.
\end{proof}

We see that it is not needed in the previous result that all $L\in\cP(n,2d)^*$ are moment functionals, i.e., represented by measures.

The \Cref{thm:schur} has several proofs.
By the classical result that $\cP(n,2)$ are sums of squares \cite{hilbert88} we get the following additional proof of the \Cref{thm:schur}.

\begin{proof}[Proof of the \Cref{thm:schur}]
We have that
\[L_A((c_0x_0 + \dots + c_n x_n)^2) := c^T A c\]
and
\[L_B((c_0x_0 + \dots + c_n x_n)^2) := c^T B c\]
hold for all $c=(c_0,\dots,c_n)^T\in\rset^{n+1}$.
Since $\cP(n,2)$ are sums of squares by \cite{hilbert88} and the matrices $A=(a_\alpha)_{\alpha\in\nset_0^n:|\alpha|=2}$ and $B=(b_\alpha)_{\alpha\in\nset_0^n:|\alpha|=2}$ are positive semi-definite we have $L_A,L_B\in\cP(n,2)^*$.
Hence, $T_A$ with diagonal sequence $A$ and $T_B$ with diagonal sequence $B$ are positivity preservers by \Cref{prop:borceaDiagHom} (ii) $\Rightarrow$ (i).
Therefore, $T_A T_B = T_{A\circ B}$ is a positivity preserver and by \Cref{prop:borceaDiagHom} (i) $\Rightarrow$ (ii) we have $A\circ B \succeq 0$.
\end{proof}

In summary, we see that the \Cref{thm:schur} is a simple consequence of the fact that the composition of two diagonal positivity preservers is again a diagonal positivity preserver.

\section{Infinitely Divisible Measures with Respect to $\odot$}
\label{sec:infdivOdot}

In \cite{didio24posPresConst} we have seen the importance of infinitely divisible measures with respect to the (additive) convolution $*$ to study positivity preservers.
We therefore look at infinitely divisible measures with respect to the (multiplicative) convolution $\odot$ to study diagonal positivity preservers.
We define the following.

\begin{dfn}
Let $n\in\nset$ and let $\mu$ be a measure on $\rset^n$.
We call $\mu$ \emph{divisible by $k\in\nset$} with respect to $\odot$ iff there exists a measure $\nu$ on $\rset^n$ such that $\mu = \nu^{\odot k}$.
We call $\mu$ \emph{infinitely divisible} with respect to $\odot$ iff $\mu$ is divisible by any $k\in\nset$.
\end{dfn}

We give some examples of infinitely divisible measures with respect to $\odot$.

\begin{exm}[$e^{tA}$ with $A = x\partial_x$]\label{exm:xD}
The operator $e^{tx\partial_x}$ with $t\in\rset$ is given by
\[e^{tx\partial_x}x^k = e^{tk}x^k,\]
for all $k\in\nset_0$ and $\lambda_t = (e^{tk})_{k\in\nset_0}$ is a moment sequence represented by $\mu_t = \delta_{e^t}$.
The measure $\mu_t$ is infinitely divisible with respect to $\odot$ since
\[\mu_{t_1}\odot \mu_{t_2} = \delta_{e^{t_1}} \odot \delta_{e^{t_2}} = \delta_{e^{t_1}\cdot e^{t_2}} = \delta_{e^{t_1+t_2}} = \mu_{t_1+t_2}\]
holds by \Cref{lem:odotProps} (\ref{item:last2}) for all $t_1,t_2\in\rset$.
\exmsymbol
\end{exm}

\begin{exm}[$e^{tA}$ with $A=(x\partial_x)^2$]\label{exm:xDsquare}
The operator $e^{t(x\partial_x)^2}$ with $t\in [0,\infty)$ is given by
\[e^{t(x\partial_x)^2} x^k = e^{tk^2} x^k\]
for all $k\in\nset_0$ and $\lambda_t = (e^{tk^2})_{k\in\nset_0}$ is a moment sequence represented by $\mu_t$ with
\[\diff\mu_t(x) := \chi_{(0,\infty)}\cdot \frac{1}{\sqrt{4\pi t}} \cdot x^{-1}\cdot \exp\left( -\frac{1}{4t}\cdot (\ln x)^2 \right) ~\diff x\]
for $t\in (0,\infty)$ and $\mu_0 = \delta_1$.
For $t=\frac{1}{2}$ this is the first explicit indeterminate moment sequence $\lambda_{\frac{1}{2}} = (e^{\frac{1}{2} k^2})_{k\in\nset_0}$ \cite[pp.\ J.106--J.107, §56]{stielt94}.
The measures $\mu_t$ are infinitely divisible since
\[\mu_{t_1}\odot\mu_{t_2} = \mu_{t_1+t_2}\]
holds for all $t_1,t_2\in [0,\infty)$.
\exmsymbol
\end{exm}

\begin{exm}
Let $n\in\nset$ and $\nu$ be a measure on $\rset^n$ with $\nu(\rset^n)<\infty$.
Then $\nu^{\odot k}$ is a measure on $\rset^n$ by \Cref{lem:odotProps}.
Then for any $t\in [0,\infty)$ we have that
\[e^{\odot t\nu} := \sum_{k\in\nset_0} \frac{t^k\cdot \nu^{\odot k}}{k!}\]
is an infinitely divisible measure on $\rset^n$ since
\[e^{\odot t_1\nu} \odot e^{\odot t_2\nu} = e^{\odot (t_1+t_2)\nu}\]
holds for all $t_1,t_2\in [0,\infty)$.
If all moments
\[s_\alpha := \int_{\rset^n} x^\alpha~\diff\nu(x)\]
of $\nu$ with $\alpha\in\nset_0^n$ exist then
\[\int_{\rset^n} x^\alpha~\diff (e^{\odot t\nu})(x) = e^{t\cdot s_\alpha}\]
holds for all $\alpha\in\nset_0^n$ and $t\in [0,\infty)$.
\exmsymbol
\end{exm}

L\'evy measures $\nu$ fulfill $\nu(\{0\}) = 0$.
The next two examples are devoted to the cases of measures with $\nu(\{0\}) > 0$.
At first for the additive convolution $*$ and then for the multiplicative convolution $\odot$.

\begin{exm}\label{exm:nu0neq0Add}
Let $n\in\nset$ and $t\in [0,\infty)$ (or all $\rset$).
Since $\delta_0 * \mu = \mu$ holds for all measures $\mu$ on $\rset^n$, i.e., $\delta_0^{*k} = \delta_0$ for all $k\in\nset_0$, we have that 
\[e^{*t\delta_0} = \sum_{k\in\nset_0} \frac{t^k}{k!}\cdot \delta_0^{*k} = e^t\cdot\delta_0 \qquad\text{and}\qquad e^{*(t\delta_0+\mu)} = e^{*t\delta_0}*e^{*\mu}=e^t\cdot e^{*\mu}\]
hold for all $t\in\rset$.
For L\'evy processes one therefore needs L\'evy measures $\nu$ to fulfill $\nu(\{0\}) = 0$ to ensure that the L\'evy process remains a probability distribution, i.e., $e^{*t\nu}(\rset^n) = 1$ for all $t\in [0,\infty)$.
If $\nu(\{0\}) \neq 0$ then this induces a scaling of the measure.
That is not allowed in probability theory but for moment measures it is allowed since the set of moment sequences is a cone.
\exmsymbol
\end{exm}

\begin{exm}\label{exm:nu0neq0Mult}
Let $n\in\nset$ and $t\in [0,\infty)$ (or all $\rset$).
Since $\delta_0\odot\mu = \mu(\rset^n)\cdot\delta_0$ holds for all measures $\mu$ on $\rset^n$ by \Cref{lem:odotProps} (v), $\delta_\one\odot\mu = \mu$ holds by \Cref{lem:odotProps} (vi), and $\mu^{\odot 0} = \delta_{\one}$ holds by \Cref{dfn:odot} we have that 
\[e^{\odot t\delta_0} = \delta_\one + \sum_{k\in\nset} \frac{t^k}{k!}\cdot\delta_0 = \delta_\one + (e^t - 1)\cdot\delta_0\]
with moments
\[\int_{\rset^n} x^\alpha~\diff\!\left(e^{\odot t\delta_0}\right)\!(x) = \begin{cases}
e^t & \text{for}\ \alpha=0\\ 1 & \text{for}\ \alpha\neq 0.\end{cases}\]
Hence, we have $e^{\odot (t\delta_0 + \mu)} = e^{\odot t\delta_0}\odot e^{\mu}$ with moments
\[\int_{\rset^n} x^\alpha~\diff\!\left(e^{\odot (t\delta_0 + \mu)} \right)\!(x) = e^{\delta_{\alpha,0}\cdot t}\cdot \exp\left(\int_{\rset^n} x^\alpha~\diff\mu(x)\right)\]
where $\delta_{\alpha,0}$ is the Kronecker delta.
Hence, $e^{\odot (t\delta_0 + \mu)}$ is no probability measure for $t\neq 0$ but rescaling gives the probability measure
\[e^{-t}\cdot e^{\odot (t\delta_0 + \mu)} = e^{-t\delta_\one}\odot e^{\odot (t\delta_0+\mu)} = e^{\odot(t\delta_0 - t\delta_\one + \mu)}.\]
Since the set of moment measures resp.\ sequences is a cone scaling is allowed and hence we can work with $e^{\odot( t\delta_0 + \mu)}$ instead of $e^{\odot (t\delta_0 - t\delta_\one + \mu)}$, i.e., in the L\'evy measures $\nu$ we can allow $\nu(\{0\})>0$.
Additionally, $e^{\odot (t\delta_0 - t\delta_\one + \mu)}$ requires the in general signed measure $t\delta_0 - t\delta_\one + \mu$ which enormously complicates things.
\exmsymbol
\end{exm}

In summary, \Cref{exm:nu0neq0Add} and \ref{exm:nu0neq0Mult} show that for the additive and multiplicative convolution $*$ and $\odot$ we can not only allow $\nu(\{0\}) >0$ for the L\'evy measures $\nu$ but that this actually simplifies our treatment.
One has to pay special attention to $\nu(\rset^n)=\infty$ since $\delta_0\odot\nu = \nu(\rset^n)\cdot\delta_0$ holds by \Cref{lem:odotProps} (v).

The reader will of course recognize the resemblance between this discussion about L\'evy measures and (un)bounded operators.
For a bounded operator $B$ on a Banach or a Hilbert space
\[e^B = \sum_{k\in\nset_0} \frac{B^k}{k!}\]
is well defined.
For unbounded operators $H$ its exponential $e^H$ can also be well-defined, but not via the Taylor sum.
The same appears here with L\'evy measures $\nu$ in the bounded ($\nu(\rset^n)<\infty$) and unbounded ($\nu(\rset^n)=\infty$) case.

\section{Diagonal Operators and their Generators}
\label{sec:generators}

Recall the definitions
\[\fD := \left\{ T = \sum_{\alpha\in\nset_0^n} q_\alpha\cdot\partial^\alpha \;\middle|\; q_\alpha\in\rset[x_1,\dots,x_n]_{\leq |\alpha|},\alpha\in\nset_0^n,\ \text{and}\ \ker T = \{0\}\right\}\]
and
\[\fd := \left\{ \sum_{\alpha\in\nset_0^n} q_\alpha\cdot\partial^\alpha \;\middle|\; q_\alpha\in\rset[x_1,\dots,x_n]_{\leq |\alpha|},\alpha\in\nset_0^n\right\}.\]
from \cite[Def.\ 1.1]{didio24posPres2arxiv}.
For diagonal operators we therefore define the following.

\begin{dfn}\label{dfn:fDdfdd}
Let $n\in\nset$.
We define subsets $\fD',\fd'\subseteq\rset[[x_1\partial_1,\dots,x_n\partial_n]]$ by
\[\fD' := \left\{\begin{gathered} T: \rset[x_1,\dots, x_n] \to \rset[x_1,\dots, x_n]\\ \text{linear operator}\end{gathered}\ \middle|\ \begin{gathered} Tx^\alpha = t_\alpha x^\alpha\ \text{with}\\ t_\alpha > 0\ \text{for all}\ \alpha \in \nset_0^n\end{gathered}\right\}\]
and
\[\fd' := \left\{\begin{gathered} A: \rset[x_1,\dots, x_n] \to \rset[x_1,\dots, x_n]\\	\text{linear operator}\end{gathered}\ \middle|\ \begin{gathered} Ax^\alpha = a_\alpha x^\alpha\ \text{with}\\ a_\alpha \in \rset\ \text{for all}\ \alpha \in \nset_0^n\end{gathered}\right\}.\]
Furthermore, we define the set of \emph{diagonal positivity preservers} with strictly positive eigenvalues as
\[\fD_+' := \left\{T \in \fD'\ \middle|\ T \text{ is a positivity preserver}\right\}\]
and their \emph{generators} as
\[\fd_+':=\left\{A \in \fd'\ \middle|\ e^{tA} \in \fD_+' \text{ for all }t \geq 0\right\}.\]
\end{dfn}

Recall that $(\fD,\,\cdot\,)$ is a regular Fr\'echet Lie group with Fr\'echet Lie algebra $(\fd,\,\cdot\,,+)$, see \cite[Thm.\ 3.7]{didio24posPres2arxiv}.
For the definition of (regular) Fr\'echet Lie group see \cite[p.\ 63, Dfn.\ 1.1]{omori97} or e.g.\ \cite[Dfn.\ 2.10]{didio24posPres2arxiv}.

\begin{thm}\label{thm:LieSubGroup}
$(\fD',\,\cdot\,)$ is a commutative regular Fr\'echet Lie sub-group of $(\fD,\,\cdot\,)$ with the Fr\'echet Lie sub-algebra $(\fd',\,\cdot\,,+)$ of $(\fd,\,\cdot\,,+)$.
The exponential map
\[\exp:\fd'\to \fD',\quad A\mapsto \exp A := \sum_{k\in\nset_0} \frac{A^k}{k!}\]
is smooth, bijective, and fulfills
\[(\exp A) x^\alpha = e^{a_\alpha} x^\alpha\]
for all $A\in\fd'$, i.e., $Ax^\alpha = a_\alpha x^\alpha$ for all $\alpha\in\nset_0^n$.
Its inverse $\log:\fD'\to\fd'$ is smooth and explicitly given by
\[(\log T)x^\alpha = \ln t_\alpha\cdot x^\alpha\]
for all $T\in\fD'$, i.e., $Tx^\alpha = t_\alpha x^\alpha$ for all $\alpha\in\nset_0^n$.
\end{thm}
\begin{proof}
It is clear that $\fD'\subsetneq\fD$ and $\fd'\subsetneq\fd$ are commutative and closed in the Fr\'echet topology.
Hence, $\fd'\subseteq\fd$ is closed with respect to the Lie bracket.
Since every diagonal operator $T$ resp.\ $A$ is uniquely determined by their diagonal sequence $(t_\alpha)_{\alpha\in\nset_0^n}$ resp.\ $(a_\alpha)_{\alpha\in\nset_0^n}$ and $\exp:\rset\to (0,\infty)$ is smooth and bijective with smooth inverse $\ln:(0,\infty)\to\rset$ we have that $\exp:\fd'\to\fD'$ is smooth with smooth explicit inverse $\log:\fD'\to\fd'$.
\end{proof}

We can now fully characterize the generators of diagonal positivity preservers.

\begin{thm}\label{thm:diagonalGenerators}
Let $n\in\nset$ and let
\[A=\sum_{\alpha \in \nset_0^n} \frac{c_\alpha}{\alpha!}\cdot x^\alpha\cdot\partial^\alpha\]
be with $c_\alpha \in \rset$ for all $\alpha \in \nset_0^n$.
Then the following are equivalent:
\begin{enumerate}[(i)]
\item $A \in \fd_+'$, i.e., $e^{tA}$ is a diagonal positivity preserver for all $t \geq 0$.

\item There exist a constant $c_0\in\rset$, a vector $b = (b_1,\dots,b_n)^T \in \rset^n$, a real symmetric positive semi-definite matrix $\Sigma = (\sigma_{i,j})_{i,j=1}^n$, and a $\sigma$-finite measure $\nu$ on $\rset^n$ with
\[
\nu(\{0\}) = 0
%
%
\quad\text{and}\quad
\int_{\rset^n} |x^\alpha| ~\diff\nu(x) < \infty
\end{equation*}
for 
all $\alpha \in \nset_0^n$ with $|\alpha| \geq 2$ such that
\begin{alignat*}{2}
c_0			&\in\rset,\\
c_{e_i}		&= b_i + \int_{\|x\|_2 \geq 1} x_i ~\diff \nu(x) &\qquad& \text{for all}\ i = 1, \dots, n,\\
c_{e_i+e_j}	&= \sigma_{i, j} + \int_{\rset^n} x^{e_i+e_j} ~\diff\nu(x)	&& \text{for all}\ i, j = 1, \dots, n,
\intertext{and}
c_\alpha 	&= \int_{\rset^n} x^\alpha ~\diff\nu(x) && \text{for all}\ \alpha \in \nset_0^n\ \text{with}\ |\alpha| \geq 3.
\end{alignat*}
\end{enumerate}
\end{thm}
\begin{proof}
(i) $\Rightarrow$ (ii): 
Since $A$ is a generator of a positive semi-group in the non-constant coefficient case we can apply \Cref{thm:posGenNonConst} ``(i) $\Rightarrow$ (ii)'' with scaling $A = c_0 + \tilde{A}$, i.e., $e^A = e^{c_0}\cdot e^{\tilde{A}}$.
With $y = \one$ we get $c_\alpha y^\alpha = c_\alpha$ and the assertion (ii) is then proved.

(ii) $\Rightarrow$ (i):
For $y \in (\rset\setminus\{0\})^n$ it follows that
\begin{alignat*}{2}
y^0c_0				&= c_0\in\rset,\\
y^{e_i} c_{e_i}		&= y_ib_i + \int_{||x||_2 \geq 1} x_i ~\diff(\nu \odot \delta_y)(x) &\quad& \text{for all}\ i = 1, \dots, n,\\
y^{e_i+e_j}c_{e_i+e_j}	&= y_i y_j\sigma_{i,j} + \int_{\rset^n} x^{e_i+e_j}~\diff(\nu \odot \delta_y)(x) &&  \text{for all}\ i,j = 1,\dots,n,
\intertext{and}
y^\alpha a_\alpha 	&= \int_{\rset^n} x^\alpha ~\diff(\nu \odot \delta_y)(x) && \text{for all}\ \alpha \in \nset_0^n\ \text{with}\ |\alpha| \geq 3.
\end{alignat*}
Here, $b(y) := (b_1 y_1, \dots, b_n y_n)^T \in \rset^n$ is a vector, $\Sigma(y) := (y\cdot y^T) \odot \Sigma$ is a positive semi-definite matrix as the Hadamard product of two positive semi-definite matrices (\Cref{thm:schur}) and $\nu_y := \nu \odot \delta_y$ is a measure on $\rset^n$ with
\[\int_{\|x\|_2 \geq 1} |x_i| ~\diff(\nu \odot \delta_y)(x) = |y_i|\cdot\int_{\|x\|_2 \geq 1} |x_i| ~\diff\nu(x) < \infty\]
and
\[\int_{\rset^n} |x^\alpha| ~\diff(\nu \odot \delta_y)(x)	= |y^\alpha|\cdot \int_{\rset^n} |x^\alpha| ~\diff\nu(x) < \infty.\]
Hence, $A_y$ is a generator of a positivity preserving semi-group.

If $y_i = 0$ for some $i=1,\dots,n$ then we find $y^{(k)}\in (\rset\setminus\{0\})^n$ with $y^{(k)}\to y$ as $k\to\infty$.
Since the set of constant coefficient generators is closed in the Fr\'echet topology \cite{didio24posPresConst} and the polynomial coefficients of $A$ are continuous, we have that $A_{y^{(k)}}\to A_y\in\fd_+'$. 
Hence, \Cref{thm:posGenNonConst} is applicable again and it follows
\[A = \sum_{\alpha \in \nset_0^n} \frac{c_\alpha}{\alpha!}\cdot x^\alpha\partial^\alpha \in \fd_+.\qedhere\]
\end{proof}

Since we characterized positivity preserving diagonal operators and their generators we can now apply these results to moment sequences.
The following was already defined by Tyan \cite[Ch.\ 4]{tyan75}.

\begin{dfn}
Let $n\in\nset$ and $t=(t_\alpha)_{\alpha\in\nset_0^n}$ be a moment sequence with $t_\alpha>0$ for all $\alpha\in\nset_0^n$.
The moment sequence $t$ is called \emph{infinitely divisible with respect to $\odot$} if $t^c :=(t_\alpha^c)_{\alpha\in\nset_0^n}$ is a moment sequence for all $c>0$.
\end{dfn}

Infinitely devisible moment sequences and diagonal operators are connected through the following.

\begin{lem}\label{lem:t_inf_div_iff_A_in_fd+}
Let $n\in\nset$ and let $t = (t_\alpha)_{\alpha \in \nset_0^n}$ be a moment sequence with $t_\alpha > 0$ for all $\alpha \in \nset_0^n$. 
Let $A: \rset[x_1,\dots,x_n] \to \rset[x_1,\dots,x_n]$ be the linear operator defined by 
\[Ax^\alpha = \ln t_\alpha\cdot x^\alpha\]
for all $\alpha \in \nset_0^n$. 
Then the following are equivalent:
\begin{enumerate}[(i)]
\item The moment sequence $t$ is infinitely divisible with respect to $\odot$.

\item $A \in \fd_+'$, i.e. $e^{cA}$ is a positivity preserver for all $c \geq 0$.
\end{enumerate}
\end{lem}
\begin{proof}
(i) $\Rightarrow$ (ii):
Using the exponential map from \Cref{thm:LieSubGroup} it follows that 
\[
	\exp(cA)x^\alpha = e^{c\ln t_\alpha}x^\alpha = t_\alpha^c x^\alpha
\]
holds for all $\alpha \in \nset_0^n$ and $c \geq 0$.
Since $t^c$ is a moment sequence for all $c \geq 0$ and using \Cref{thm:borceaDiag} it follows that $\exp(cA)$ is a positivity preserver for all $c \geq 0$, i.e., $A \in \fd_+'$ by \Cref{dfn:fDdfdd}.

(ii) $\Rightarrow$ (i):
Similarly, since $\exp(cA)$ is a positivity preserver for all $c \geq 0$, using \Cref{thm:borceaDiag} gives a representing measure $\mu_c$ of $t^c$ for each $c \geq 0$.
Hence, $t^c$ is a moment sequence for all $c \geq 0$ and $t$ is infinitely divisible.
\end{proof}

We can now use the connection between infinitely divisible moment sequences and diagonal positivity preservers to characterize infinitely divisible moment sequences without the L\'evi--Khinchin formula for $\odot$-divisible measures.

\begin{thm}\label{thm:infDivRn}
Let $n\in\nset$ and let $t = (t_\alpha)_{\alpha\in\nset_0^n}$ be a moment sequence.
Then the following are equivalent:
\begin{enumerate}[(i)]
\item The moment sequence $t$ is infinitely divisible with respect to $\odot$.

\item There exist a constant $c_0 \in \rset$, a vector $b = (b_1, \dots, b_n)^T \in \rset^n$, a real symmetric positive semi-definite matrix $\Sigma = (\sigma_{i, j})_{i, j = 1}^n \succeq 0$, and a $\sigma$-finite measure $\nu$ on $\rset^n$ with
\[
\nu(\{0\})=0 
%
%
\quad\text{and}\quad
\int_{\rset^n} |x^\alpha| ~\diff\nu(x) < \infty
\]
for 
all $\alpha \in \nset_0^n$ with $|\alpha| \geq 2$ such that 
\[
	t_\alpha = e^{C_\alpha + I_\alpha}
\]
holds with
\begin{align}
C_\alpha &:= c_0 + \sum_{i = 1}^n \alpha_ib_i 
	+ \sum_{\substack{i, j=1\\i \neq j}}^n \alpha_i\alpha_j\sigma_{i, j}
	+ \frac{1}{2}\sum_{i=1}^n \alpha_i(\alpha_i-1)\sigma_{i, i}\label{eq:CalphaDef}
\intertext{and}
I_\alpha
&:= \int_{\rset^n} (x+\one)^\alpha - 1
- \sum_{i = 1}^n \alpha_i\cdot x_i\cdot \chi_{\{\|x\|_2 < 1\}} ~\diff \nu(x)\label{eq:IalphaDef}
\end{align}
for all $\alpha \in \nset_0^n$.	
\end{enumerate}
\end{thm}
\begin{proof}
By \Cref{thm:diagonalGenerators} and \Cref{lem:t_inf_div_iff_A_in_fd+} we have that (i) is equivalent to have an operator $A \in\fd_+'$ with
\[Ax^\alpha = \ln t_\alpha \cdot x^\alpha\]
for all $\alpha \in \nset_0^n$, a constant $c_0\in\rset$, a vector $b = (b_1,\dots,b_n)^T \in \rset^n$, a real symmetric positive semi-definite matrix $\Sigma = (\sigma_{i,j})_{i,j=1}^n \succeq 0$, and a $\sigma$-finite measure $\nu$ on $\rset^n$ with	
\[\nu(\{0\})=0 
%
\quad\text{and}\quad
\int_{\rset^n} |x^\alpha| ~\diff\nu(x) < \infty\]
for 
all $\alpha \in \nset_0^n$ with $|\alpha| \geq 2$ such that 
\[A = \sum_{\alpha \in \nset_0^n} \frac{c_\alpha}{\alpha!}\cdot x^\alpha\partial^\alpha\]
with 
\begin{alignat*}{2}
c_0	&\in\rset,\\
c_{e_i}		&= b_i + \int_{\|x\|_2 \geq 1} x_i ~\diff \nu(x) &\qquad& \text{for all}\ i = 1, \dots, n,\\
c_{e_i+e_j}	&= \sigma_{i, j} + \int_{\rset^n} x^{e_i+e_j} ~\diff\nu(x)	&& \text{for all}\ i, j = 1, \dots, n,
\intertext{and}
c_\alpha 	&= \int_{\rset^n} x^\alpha ~\diff\nu(x) && \text{for all}\ \alpha \in \nset_0^n\ \text{with}\ |\alpha| \geq 3
\end{alignat*}
holds.
By \Cref{rem:diagOpCoefficients} we have
\begin{align*}
\ln t_\alpha = \sum_{\beta \preceq \alpha} \binom{\alpha}{\beta}\cdot c_\beta
&= c_0 + \sum_{i = 1}^n \alpha_ib_i 
			+ \sum_{\substack{i, j=1\\i \neq j}}^n \alpha_i\alpha_j\sigma_{i, j}
			+ \frac{1}{2}\sum_{i=1}^n \alpha_i(\alpha_i-1)\sigma_{i, i}
\\		&\qquad
			+ \sum_{i = 1}^n \int_{\|x\|_2 \geq 1} \alpha_i x_i ~\diff\nu(x)
			+ \sum_{\substack{\beta \preceq \alpha\\ |\beta| \geq 2}} \binom{\alpha}{\beta} \int_{\rset^n} x^\beta ~\diff \nu(x)
\end{align*}
for all $\alpha \in \nset_0$. Defining $C_\alpha$ as in (\ref{eq:CalphaDef}) and using the multi-binomial theorem
\[
	(x+\one)^\alpha = \sum_{\beta \preceq \alpha} \binom{\alpha}{\beta}x^\beta
	= 1 + \sum_{i = 1}^n \alpha_i x_i 
	+ \sum_{\substack{\beta \preceq \alpha\\ |\beta| \geq 2}} \binom{\alpha}{\beta} x^\beta
\]
inside the integral it follows that
\begin{align*}
	\ln t_\alpha 
	&= C_\alpha + \int_{\rset^n} (x+\one)^\alpha - 1 + \sum_{i = 1}^n \alpha_i\cdot x_i\cdot (\chi_{\{\|x\|_2 \geq 1\}}(x) - 1) ~\diff \nu(x)
\intertext{and replacing the characteristic funtion $\chi$ with $\|x\|_2\geq 1$ by $\|x\|_2< 1$ gives}
	&= C_\alpha + \int_{\rset^n} (x+\one)^\alpha - 1
		- \sum_{i = 1}^n \alpha_i\cdot x_i\cdot\chi_{\{\|x\|_2 < 1\}} ~\diff \nu(x)
\end{align*}
for all $\alpha \in \nset_0^n$.
With $I_\alpha$ defined as in (\ref{eq:IalphaDef}) we arrive at 
\[
	t_\alpha = e^{\ln t_\alpha} = e^{C_\alpha + I_\alpha}
\]
for all $\alpha \in \nset_0^n$, i.e., the operator $A$ with the measure $\nu$ is equivalent to (ii).
\end{proof}

For the special case $n=1$ we get the following result, compare with \cite[Thm.\ 4.2]{tyan75}.
Our approach is faster then in \cite[Thm.\ 4.2]{tyan75}.

\begin{cor}\label{cor:tHochcR}
Let $t = (t_k)_{k\in\nset_0}\subseteq (0,\infty)$.
Then the following are equivalent:
\begin{enumerate}[(i)]
\item $t^c$ is a moment sequence for all $c>0$.

\item There exist constants $a,b,c\in\rset$ with $a\geq 0$ and a $\sigma$-finite measure $\nu$ on $\rset$ with
\[\nu(\{0\}) = 0 \quad\text{and}\quad \int_\rset |x^i|~\diff\nu(x) < \infty\]
for all $i\in\nset$ with $i\geq 2$ such that
\[\ln t_k = c + b\cdot k + a\cdot k^2 + \int_\rset (x+1)^k -1 -k\cdot x\cdot \chi_{(-1,1)}(x)~\diff\nu(x)\]
holds for all $k\in\nset_0$.
\end{enumerate}
\end{cor}
\begin{proof}
Set $n=1$ in \Cref{thm:infDivRn}, adjust $c_0,b,\sigma$ in \Cref{thm:infDivRn} to $a,b,c$ here, and $\int_\rset x^2~\diff\nu(x) < \infty$ implies $\int_{|x|\geq 1} |x|~\diff\nu(x)$.
\end{proof}

For the Stieltjes case, i.e., on $(0,\infty)$, we get the following.

\begin{cor}\label{cor:tHochcStielt}
Let $t = (t_k)_{k \in \nset_0}\in\rset^{\nset_0}\setminus\{0\}$.
Then the following are equivalent:
\begin{enumerate}[(i)]
\item $t^c$ is a Stieltjes moment sequence for all $c > 0$.

\item There exist $a,b,c\in\rset$ with $a\geq 0$ and a $\sigma$-finite measure $\nu$ on $[0,\infty)$ with
\[
\nu(\{0\})=0
\quad\text{and}\quad
\int_0^\infty |x^i| ~\diff\nu(x) < \infty
\]
for all $i\in\nset$ with $i \geq 2$ such that 
\[
\ln t_k = c + b\cdot k + a\cdot k^2 + \int_0^\infty (x+1)^k - 1 - k\cdot x\cdot \chi_{[0,1)}(x) ~\diff \nu(x)
\]
holds for all $k \in \nset_0$.
\end{enumerate}
\end{cor}
\begin{proof}
By including the restriction $\supp\mu\subseteq [0,\infty)$ in the proof of \cite[Main Thm.\ 4.7]{didio24posPresConst} we find that in (i) $t$ is represented by a Stieltjes measure $\mu$, i.e., supported on $[0,\infty)$, which is infinitely divisible with respect to $*$.
Since the infinitely $*$-divisible representing measure is supported on $[0,\infty)$ we have that also its L\'evy measure $\nu$ is supported on $[0,\infty)$ in \Cref{cor:tHochcR}.
\end{proof}

Note, the L\'evy measures $\nu$ in \Cref{thm:infDivRn}, \Cref{cor:tHochcR}, and \Cref{cor:tHochcStielt} are the L\'evy measures with respect to the additive convolution $*$ since we used \Cref{thm:posGenConst} resp.\ \ref{thm:posGenNonConst}.
Interestingly, going from \Cref{cor:tHochcR} to \ref{cor:tHochcStielt} one restricts the L\'evy measure $\nu$ from $\rset$ to $[0, \infty)$ but the contribution $a\cdot k^2$ with $a\geq 0$ remains.
In \Cref{exm:xDsquare} we have seen that $a\cdot k^2$ corresponds to
\[a\cdot (x\partial_x)^2 = a\cdot( x^2\partial_x^2 + x\partial_x),\]
i.e., the $\log$-normal distribution.
In \Cref{exm:xD} we found that $x\partial_x$ corresponds to $\delta_{e^{t}}$, $t\in\rset$.
In summary, while the L\'evy measure $\nu$ comes from the additive convolution $*$ in \Cref{thm:posGenConst} for
\[A = c + b\cdot\partial_x + a\cdot\partial_x^2 + \dots \tag{$n=1$}\]
the same L\'evy measure appears in \Cref{thm:infDivRn}, \Cref{cor:tHochcR}, and \Cref{cor:tHochcStielt} for
\[A = c + (b+a)\cdot x\partial_x + a\cdot x^2\partial_x^2 + \dots \tag{$n=1$}\]
for the multiplicative convolution $\odot$.

Both Berg \cite{berg04,berg05,berg07} and Tyan \cite{tyan75} have studied these infinitely $\odot$-divisible moment sequences before.
While Berg \cite{berg07} was focused on Stieltjes moment sequences giving \Cref{cor:tHochcStielt}, Tyan proved a more general case for all multivariate moment sequences giving \Cref{thm:infDivRn}. 
The proofs of \cite[Thm.\ 2.4]{berg07}, \cite[Thm.\ 4.2]{tyan75} and \cite[Main Thm.\ 4.11]{didio24posPresConst} all use measure semi-groups and their generators. 
While Berg and Tyan use the techniques from Hunt's theorem \cite{hunt56} to prove a L\'evy-Khinchine type formula on moment sequences, in \cite{didio24posPresConst} we used the differential operator sum representation and the known L\'evy-Khinchine representation. 
This describes the coefficients of the semi-group generating operator in the differential operator sum representation.
These coefficients are obtained in \Cref{thm:posGenNonConst} and in \Cref{thm:diagonalGenerators}. 
The coefficients from \Cref{thm:diagonalGenerators} could also have been obtained by using Tyan's result for the diagonal values and then \Cref{rem:diagOpCoefficients} to convert the diagonal values to the respective coefficients.

\section{Summary}
\label{sec:summary}

In this work we continue our studies about positivity preservers \cite{didio24posPresConst,didio24posPres2arxiv} and we focus here on diagonal operators since they possess properties the general positivity preservers do not possess.
At first we characterize in \Cref{thm:linS} all linear maps $S:\rset^{\nset_0^n}\to\rset^{\nset_0^n}$ which preserve $K$-moment sequences and hence we generalize \cite{carcoma17}.

In \Cref{sec:diagOp} we start our investigation of diagonal operators by looking at different types of representation of diagonal operators, i.e., we have the three representations
\[(\text{a})\ Tx^\alpha = t_\alpha x^\alpha, \quad (\text{b})\ T=\sum_{\alpha\in\nset_0^n} \frac{c_\alpha}{\alpha!}\cdot x^\alpha\cdot\partial^\alpha, \quad\text{and}\quad (\text{c})\ T=\sum_{\alpha\in\nset_0^n} \frac{d_\alpha}{\alpha!}\cdot (x\partial)^\alpha\]
and the sequences $(t_\alpha)_{\alpha\in\nset_0^n}$, $(c_\alpha)_{\alpha\in\nset_0^n}$, and $(d_\alpha)_{\alpha\in\nset_0^n}$ transform into each other by binomial coefficients and Stirling numbers, see \Cref{rem:diagOpCoefficients}.
Diagonal operators always have a representation (a) and (b).
In \Cref{exm:dalphaNonEx} we give a diagonal positivity preserver which can not be represented by type (c).

In \Cref{sec:prodDiag} we see that the Hadamard product $s\odot t = (s_\alpha\cdot t_\alpha)_{\alpha\in\nset_0^n}$ of two sequences corresponds to the product $ST$ of two diagonal operators and the multiplicative convolution $\mu\odot\nu$ of their representing measures.
This observation simplifies several previous proofs, e.g.\ by removing the Schur product arguments in \cite{belton22} and \cite{blekherman22arxiv}.
In fact, the simplification provides us with another proof of the \Cref{thm:schur} in \Cref{sec:schur}.

In \Cref{sec:infdivOdot} we deal with $\odot$-infinitely divisible measures.
These properties are used in \Cref{sec:generators} to prove \Cref{thm:diagonalGenerators}: the full characterization of generators $A$ of diagonal positivity preservers, i.e., $e^{tA}$ is a diagonal positivity preserver for all $t\geq 0$.
This characterization is used to give an alternative (and shorter) approach to characterize infinitely divisible moment sequences \cite{horn69b,horn69a,tyan75} and reveals the connection between infinitely divisible moment sequences and the generators of diagonal positivity preservers.

\section*{Funding}

The authors and this project are supported by the Deutsche Forschungs\-gemein\-schaft DFG with the grant DI-2780/2-1 of PdD and the research fellowship of PdD at the Zukunfts\-kolleg of the University of Konstanz, funded as part of the Excellence Strategy of the German Federal and State Government.
\vspace{-.3cm}


\newcommand{\etalchar}[1]{$^{#1}$}
\providecommand{\bysame}{\leavevmode\hbox to3em{\hrulefill}\thinspace}
\providecommand{\MR}{\relax\ifhmode\unskip\space\fi MR }
\providecommand{\MRhref}[2]{%
  \href{http://www.ams.org/mathscinet-getitem?mr=#1}{#2}
}
\providecommand{\href}[2]{#2}

\end{document}